\documentclass[12pt]{amsart}
\usepackage{amsmath}
\usepackage{amsthm}
\usepackage{amssymb}
\usepackage{tikz}
\usepackage[all]{xy}
\usepackage[left=2.5cm,right=2.5cm,top=3cm,bottom=3cm]{geometry}
\usepackage{eucal}
\usepackage{palatino}
\usepackage{euler}
\usepackage{mathrsfs}
\usepackage{amssymb}
\usepackage{amscd}
\usepackage{latexsym}
\usepackage{epsfig}
\usepackage{graphicx}
\usepackage{amsfonts}
\usepackage{psfrag}
\usepackage{caption}
\usepackage{rotating}
\usepackage{mathtools}
\usepackage{fullpage}
\usepackage{etoolbox} 

\usepackage[normalem]{ulem}

\usepackage{pifont}
\usepackage{epsfig}

\usepackage{url}
\usepackage{cite}
\usepackage{dsfont}
\usepackage{array}

\newtheorem{definition}{Definition}
\newtheorem{theorem}{Theorem}
\newtheorem{proposition}{Proposition}

\newtheorem{lemma}{Lemma}

\theoremstyle{remark}
\newtheorem*{remark}{Remark}

\DeclareMathOperator{\GL}{GL}

\DeclareMathOperator{\Hom}{Hom}

\DeclareMathOperator{\End}{End}

\DeclareMathOperator{\Sym}{Sym}

\newtoggle{showColor}
\newtoggle{showNotes}
\toggletrue{showColor}
\togglefalse{showNotes}

\iftoggle{showColor} {%
}{%

}

\iftoggle{showNotes} {%
	\newcommand{\note}[1]{{\textcolor{red}{$\langle$#1$\rangle$}}} 
}{%
	\newcommand{\note}[1]{}
}

\title{Generalized Equivariant Cohomology and Stratifications}
\author{Peter Crooks}\email{peter.crooks@utoronto.ca}

\author{Tyler Holden}\email{tholden@math.toronto.edu}

\thanks{The first author was supported by NSERC CGS-D3 and OGS scholarships during the preparation of this work.}

\begin{document}
\begin{abstract}
For $T$ a compact torus and $E_T^*$ a generalized $T$-equivariant cohomology theory, we provide a systematic framework for computing $E_T^*$ in the context of equivariantly stratified smooth complex projective varieties. This allows us to explicitly compute $E_T^*(X)$ as an $E_T^*(\text{pt})$-module when $X$ is a direct limit of smooth complex projective $T_{\mathbb{C}}$-varieties with finitely many $T$-fixed points and $E_T^*$ is one of $H_T^*(\cdot;\mathbb{Z})$, $K_T^*$, and $MU_T^*$. We perform this computation on the affine Grassmannian of a complex semisimple group.
\end{abstract}

\maketitle

\section{Introduction}
Generalized equivariant cohomology theories have received considerable attention in the modern research literature. Particular emphasis has been placed on cohomology computations in the presence of well-behaved equivariant stratifications. Indeed, Atiyah and Bott \cite{Yang-Mills} gave an inductive procedure for computing the ordinary equivariant cohomology of a manifold in terms of the cohomologies of the strata in an equivariant stratification. Kirwan \cite{KirwanBook} then applied related ideas to a Morse-type stratification arising from the norm-square of a moment map. A paper by Harada, Henriques, and Holm \cite{HHH2005} subsequently broadened this Atiyah-Bott-Kirwan framework to include generalized equivariant cohomology calculations via infinite stratifications. This work was partly motivated by a desire to develop a GKM-type theory for the partial flag varieties of Kac-Moody groups.

Our paper has two principal objectives. The first is to provide a straightforward, self-contained account of how to perform generalized torus-equivariant cohomology computations with a finite equivariant stratification of a smooth complex projective variety. While this is readily deducible from existing work, we believe it might serve as a convenient reference for other authors. More importantly, however, it provides the context for the second of our objectives-- a computation of the generalized torus-equivariant cohomology of a direct limit of smooth projective varieties with finitely many $T$-fixed points. More specifically, we will prove the following theorem.

\begin{theorem}\label{Main Theorem}
Suppose that $T$ is a compact torus with complexification $T_{\mathbb{C}}$, and let $E_T^*$ be one of $H_T^*(\cdot;\mathbb{Z})$, $K_T^*$, and $MU_T^*$. Let $X_0\subseteq X_1\subseteq X_2\subseteq\ldots$ be a sequence of equivariant closed embeddings of smooth complex projective $T_{\mathbb{C}}$-varieties, each with finitely many $T$-fixed points. If we define $X$ to be the direct limit of the varieties $X_n$ in their classical topologies, then $$E_T^*(X)\cong\prod_{x\in X^T}E_T^*(\text{pt})$$ as $E_T^*(\text{pt})$-modules.
\end{theorem}

While much of our work was inspired by \cite{HHH2005}, there are some important distinctions to be made. In \cite{HHH2005}, the authors first work in the context of a topological group $G$ and a fairly general stratified $G$-space $X$. Among other things, they provide some conditions on the stratification which explicitly determine the generalized $G$-equivariant cohomology of $X$ in terms of the cohomologies of the strata. By contrast, we deal with stratifications only in the context of a compact torus $T$ and a smooth complex projective $T_{\mathbb{C}}$-variety $Y$. We instead try to emphasize that the task of computing the generalized $T$-equivariant cohomology of $Y$ (and also direct limits of varieties $Y$) is especially simple.

Let us briefly outline the structure of this paper. Section \ref{sec:Generalized Equivariant Cohomology} begins with a brief overview of $T$-ring spectra and how they give rise to generalized $T$-equivariant cohomology theories. Recognizing that our arguments make extensive use of equivariant Euler classes, we include a short discussion of complex oriented theories. Also included in Section \ref{sec:Generalized Equivariant Cohomology} are brief descriptions of the three theories to which we will sometimes restrict our attention: ordinary equivariant cohomology $H_T^*$, (complex) equivariant K-theory $K_T^*$, and equivariant complex cobordism $MU_T^*$.

Section \ref{sec:Cohomology and Stratifications} is devoted to understanding the $E_T^*(\text{pt})$-module structure of $E_T^*(X)$, where $X$ is a $T$-space filtered by smooth complex projective $T_{\mathbb{C}}$-varieties with finitely many $T$-fixed points. We begin with \ref{sub:Finite Stratifications}, in which Thom-Gysin sequences are used to compute the generalized $T$-equivariant cohomology of a finitely stratified smooth complex projective $T_{\mathbb{C}}$-variety. In \ref{sub:Finite Fixed Points}, we specialize to the case where our variety has finitely many $T$-fixed points and $E_T^*$ is one of $H_T^*$, $K_T^*$, and $MU_T^*$. We conclude with \ref{sub:Direct Limits of Projective Varieties}, where we generalize to the case of direct limits of the varieties considered in \ref{sub:Finite Fixed Points}. This results in Theorem \ref{Main Theorem}.

In Section \ref{sec: Affine Grassmannian}, we give an example of a $T$-space satisfying the hypotheses of Theorem \ref{Main Theorem}, namely the affine Grassmannian of a simply-connected complex semisimple group.

\subsection*{Acknowledgements}
We gratefully acknowledge the support provided by Lisa Jeffrey and Paul Selick while this work was being prepared. We also wish to thank Steven Rayan for useful discussions and for his careful reading of this manuscript.

\section{Generalized Equivariant Cohomology} 
\label{sec:Generalized Equivariant Cohomology}

\subsection{General Overview} 
In the interest of clarity, we will begin with a brief overview of the pertinent parts of a generalized equivariant cohomology theory. Let $T$ denote a fixed compact torus, and define a $T$-space to be a compactly generated weak Hausdorff topological space $X$ endowed with a continuous action of $T$. These spaces form the objects of a category $\mathcal{C}_T$, whose morphisms are $T$-equivariant continuous maps. While this is precisely the category on which we would like to define our generalized $T$-equivariant cohomology theories, some of our arguments will be more transparent in the framework of the homotopy category of $T$-equivariant spectra.

Fix a complete $T$-universe, namely a real orthogonal $T$-representation $\mathcal{U}$ of countably infinite dimension, such that $\mathcal{U}$ contains infinitely many copies of each finite-dimensional $T$-representation. Recall that $T$-spectra indexed on\footnote{We will henceforth assume that all $T$-spectra are indexed on $\mathcal{U}$.} $\mathcal{U}$ (see Definition 9.4.1 of \cite{Axiomatic}) form a category, $TS^U$ \note{What are the maps, and can we comment more on what $\Sigma^\infty$ is? Right now it just plops out of nowhere}. Also, there is a suspension functor $\Sigma^{\infty}:(\mathcal{C}_T)_*\rightarrow TS^U$, where $(\mathcal{C}_T)_*$ is the category of based $T$-spaces. In this way, based $T$-spaces yield $T$-spectra, and we will sometimes make no distinction between a based $T$-space $X$ and its suspension spectrum $\Sigma^{\infty}(X)$.

The functor $\Sigma^\infty$ is one of a family of suspension functors $(\mathcal{C}_T)_*\rightarrow TS^U$ indexed by finite-dimensional real $T$-representations. Let $V$ be one such representation, and denote by $S^V$ its one-point compactification with base point at infinity.  Note that the action of $T$ on $V$ extends to an action on $S^V$ that fixes the basepoint. Smashing against these spheres generalizes the usual suspension process, defining a functor $\Sigma^V:(\mathcal{C}_T)_*\rightarrow(\mathcal{C}_T)_*$ by
$$\Sigma^V(X):=S^V\wedge X.$$

If $V\subseteq W$ is an inclusion of finite dimensional $T$-representations, we define the relative suspension of a based $T$-space $X$ to be
$$(\Sigma^{\infty}_V(X))(W):=\Sigma^{V^{\perp}}(X),$$ where $V^{\perp}$ is the orthogonal complement of $V$ in $W$.
If $V$ does not include into $W$, we define $(\Sigma^{\infty}_V(X))(W)$ to be a point. The spaces $\{(\Sigma^{\infty}_V(X))(W)\}_W$ constitute a $T$-prespectrum and therefore determine a $T$-spectrum $\Sigma^{\infty}_V(X)$. Furthermore, $X \mapsto \Sigma^\infty_V(X)$ defines a functor
$$\Sigma^{\infty}_V:(\mathcal{C}_T)_*\rightarrow TS^U.$$ One may use this functor to define desuspensions of representation spheres:
$$S^{-V}:=\Sigma^{\infty}_V(S^0).$$
If $W$ is another finite-dimensional $T$-representation, we set
$$S^{W-V}:=S^W\wedge S^{-V}.$$
This gives us a $T$-spectrum $S^{\alpha}$ for each $\alpha$ in the representation ring $RO(T;U)$(see \cite{May}).

\subsubsection{Cohomology via spectra:} 
We have developed the machinery necessary to explain how generalized $T$-equivariant cohomology theories arise from $T$-spectra. Denote by $\overline{h}TS^U$ the stable homotopy category of $T$-spectra obtained by inverting the weak equivalences in $TS^U$. Fix a $T$-spectrum $E$, and define a functor $\tilde E_T^0: \overline hTS^U \to \mathbb Z\text{-mod}$ by associating to each $T$-spectrum $F$ the abelian group $[F,E]:=\Hom(F,E)$ of morphisms in $\overline{h}TS^U$. One may extend $\tilde{E}_T^0$ to an $RO(T;U)$-graded functor by setting
\begin{equation}\label{eq:Grading}
\tilde{E}_T^{\alpha}(F):=[S^{-\alpha}\wedge F,E], \quad \alpha \in RO(T;U).
\end{equation}

We will be primarily interested in the underlying $\mathbb{Z}$-graded functor. More explicitly, if $n\in\mathbb{Z}$, then $\tilde{E}_T^n:\overline{h}TS^U\rightarrow\mathbb{Z}\text{-mod}$ is defined via \eqref{eq:Grading} by setting $\alpha$ equal to the appropriately signed $\vert n\vert$-dimensional trivial $T$-representation. The resulting $\mathbb{Z}$-graded functor $\tilde{E}_T^*$ then restricts to a reduced generalized $T$-equivariant cohomology theory on $(\mathcal{C}_T)_*$, with the associated unreduced theory $E_T^*$ on $\mathcal{C}_T$ given by
$$E_T^*(X):=\tilde{E}_T^*(X_+).$$
Here $X_+$ is the $T$ space formed by taking a disjoint union of $X$ and an additional base point.

If $E$ is additionally a commutative $T$-ring spectrum \cite[Chapter XII]{May}, then $E_T^*$ take values of the category $\text{CRing}_{\mathbb Z}$ of $\mathbb Z$-graded commutative rings.
We then have the following definition of a generalized $T$-equivariant cohomology theory suitable for our purposes.

\begin{definition}
A generalized $T$-equivariant cohomology theory is a $\mathbb{Z}$-graded functor $E_T^*:\mathcal{C}_T\rightarrow\text{CRing}_{\mathbb{Z}}$ resulting from a commutative ring $T$-spectrum $E$ as indicated above.
\end{definition}

\subsubsection{Additional Structure:} 
Given a commutative $T$-ring spectrum $E$ and a $T$-space $X$, the map $X\rightarrow\text{pt}$ yields a morphism $E_T^*(\text{pt})\rightarrow E_T^*(X)$ of $\mathbb{Z}$-graded commutative rings. This map renders $E_T^*(X)$ a module over the ring $E_T^*(\text{pt})$.

A second consideration concerns equivariant Thom and Euler classes, and requires that we take $E_T^*$ to be a complex oriented theory \cite{CGK}. In more detail, suppose that $\xi$ is a $T$-equivariant complex vector bundle of rank $n$ over a $T$-space $X$, and let $Th(\xi)$ denote the associated Thom space. There exists a $T$-equivariant Thom class $u_T(\xi)\in\tilde{E}_T^{2n}(Th(\xi))$ which shares many of the properties of its non-equivariant counterpart, such as being natural under pullbacks and multiplicative over Whitney sums.
\note{I don't think that the properties of a such a class are known in general. Which things carry over from singular cohomology and which don't. Simply saying ``there is a Thom class'' is probably confusing.}

Associated to the Thom class is the Euler class, defined as follows: If $z: X_+ \to Th(\xi)$ is the zero section of the natural projection, define $e_T(\xi)\in E_T^{2n}(X)$ as
$$e_T(\xi):=z^*(u_T(\xi))\in\tilde{E}_T^{2n}(X_+)=E_T^{2n}(X).$$
\note{The approach here is motivated by \cite{CGK} as well as (What is this reference? It was hard coded) [6] in the bibliography of the former.}

Finally, one says that $E_T^*$ is a complex stable ring theory if for each finite-dimensional complex $T$-representation $V$, there exists a class $\alpha_{V}\in\tilde{E}_T^{\dim_{\mathbb{R}}(V)}(S^V)$ with the property that multiplication by $\alpha_V$ defines an isomorphism $\tilde{E}_T^*(X)\rightarrow \tilde{E}_T^*(S^V\wedge X)$ for all $T$-spaces $X$. Setting $X=S^0$ implies that $\tilde{E}_T^*(S^V)$ is freely generated by $\alpha_V$ as a module over $E_T^*(\text{pt})$.

We note that every complex oriented theory is a complex stable ring theory\cite{CGK}.



\subsection{Important Examples} 
Despite having discussed generalized equivariant cohomology theories in the abstract, we will sometimes emphasize three important generalized $T$-equivariant cohomology theories: (ordinary) equivariant cohomology $H_T^*$, (complex) equivariant K-theory $K_T^*$, and equivariant complex cobordism $MU_T^*$. With this in mind, it will be prudent to recall the following proposition.

\begin{proposition}
Assume that $E_T^*$ is one of $H_T^*$, $K_T^*$, and $MU_T^*$. If $V$ is a finite-dimensional complex $T$-representation, then $E_T^*(S^V)$ is free and of rank one as a module over $E_T^*(\text{pt})$, and it vanishes in odd grading degrees.
\end{proposition}

We include a brief summary of those parts of each theory that will later prove relevant.


\subsubsection{Ordinary Equivariant Cohomology} 
We denote by $ET\rightarrow ET/T=BT$ the universal principal $T$-bundle, characterized by the property that $ET$ is a contractible space on which $T$ acts freely. If $X$ is a $T$-space, then the product $X\times ET$ carries a $T$-action and we may form the Borel mixing space $$X_T:=(X\times ET)/T.$$ We then define the ordinary $T$-equivariant cohomology of $X$ (with integer coefficients) to be $$H_T^*(X):=H^*(X_T;\mathbb{Z}),$$ the integral cohomology of $X_T$. \note{Moved this up and changed how the reference is displayed} Of course, $H_T^*$ arises from the Eilenberg-MacLane $T$-spectrum \cite[Chapter XIII]{May}.

There is a natural ring isomorphism between the base ring $H_T^*(\text{pt})$ and $\Sym_{\mathbb{Z}}(X^*(T))$, the symmetric algebra of the weight lattice $X^*(T)$ of $T$. Indeed, a weight $\mu:T\rightarrow S^1$ yields an associated line bundle $$L(\mu):=\frac{ET\times\mathbb{C}}{(\alpha,z)\sim (t\alpha,\mu(t)z)}\rightarrow BT,$$ where $t\in T$ and $(\alpha,z)\in ET\times\mathbb{C}$. The ring isomorphism then associates to $\mu\in X^*(T)$ the first Chern class $c_1(L(\mu))\in H^2(BT;\mathbb{Z})=H^2_T(\text{pt})$.


\subsubsection{Equivariant K-Theory} 
Our treatment follows that given in \cite{Segal}. Recall that for a compact $T$-space $X$, $K_T^0(X)$ is defined to be the Grothendieck group of the category of $T$-equivariant complex vector bundles {over} $X$. The operation of taking the tensor product of equivariant vector bundles renders $K_T^0(X)$ a commutative ring. One extends the definition of $K_T^0$ to a definition of $K_T^n(X)$ for $X$ locally compact and $n$ any integer. By virtue of Bott periodicity, there are natural $\mathbb{Z}$-module isomorphisms $K_T^n(X)\cong K_T^{n+2}(X)$, $n\in\mathbb{Z}$. In particular, if $n\in\mathbb{Z}$, then $K_T^{2n}(\text{pt})$ is naturally isomorphic to (the underlying abelian group of) the representation ring $R(T)$ of $T$. Note that $R(T)$ is freely generated over $\mathbb{Z}$ by $\{e^{\mu}:\mu\in X^*(T)\}$, where $e^{\mu}\in R(T)$ denotes the class of the one-dimensional complex $T$-representation of weight $\mu$. Furthermore, $K_T^{2n+1}(\text{pt})=K_T^{-1}(\text{pt})=0$. Hence, we shall identify $K_T^*(\text{pt})$ as a $\mathbb{Z}$-graded abelian group with $R(T)^{\oplus 2\mathbb{Z}}$. If we multiply elements in the grading components of the latter as elements of $R(T)$, then this becomes an isomorphism of $\mathbb{Z}$-graded commutative rings. \note{This last part should be verified.}

It will later be necessary to discuss the $T$-equivariant K-theory of spaces that are not locally compact. To encompass this larger class of spaces, we will define $T$-equivariant K-theory via its ring $T$-spectrum \cite[Chapter XIV]{May}.


\subsubsection{Equivariant Complex Cobordism} 

\note{Need to write this section. The other two theories have brief intros, so it is asymmetric to leave out CC.}

Our discussion of the equivariant complex cobordism follows that of \cite{Sinha,May}.
As in Section \ref{sec:Generalized Equivariant Cohomology}, fix a complete $T$-universe $\mathcal{U}$ and let $BU^T(n)$ denote the Grassmannian of complex linear $n$-planes in $\mathcal{U}$. This Grassmannian comes equipped with a tautological line bundle $\xi_n^T \to BU^T(n)$, which is well known to serve as a model for the universal complex $n$-plane bundle. If $V$ is a finite-dimensional complex $T$-representation, let $\xi^T_V = \xi^T_{\dim_{\mathbb{C}}(V)}$. One then forms $Th(U)$, an $R(T)$-indexed pre-spectrum whose $V^{th}$ entry is $Th(\xi_V^T)$.
The spectrification of $Th(U)$ yields the spectrum $MU_T$.

\section{Cohomology and Stratifications} 
\label{sec:Cohomology and Stratifications}

Herein we examine how to deduce the $E_T^*(\text{pt})$-module structure for spaces which admit equivariant stratifications. When there are only finitely many strata, the process amounts to inductively adding strata and will terminate after finitely many steps. We explore this case further in Section \ref{sub:Projective Variety Example} 
using a natural stratification of a smooth projective $T_{\mathbb{C}}$-variety admitting finitely many $T$-fixed points.

Section \ref{sub:Direct Limits of Projective Varieties} then provides a generalization of Section \ref{sub:Projective Variety Example}, replacing smooth projective $T_{\mathbb{C}}$-varieties with direct limits thereof.


\note{Note to self: Add in some introduction here, then define $T$-equivariant stratification. Move original intro down.}

\subsection{Finite Stratifications} 
\label{sub:Finite Stratifications}

Throughout this section let $T$ be a compact torus with complexification $T_{\mathbb{C}}$, and assume that $E_T^*$ is a complex oriented generalized equivariant cohomology theory.

\begin{definition}
\label{def:Finite Stratification}
Let $X$ be a smooth complex projective variety on which $T_{\mathbb C}$ acts algebraically. A \emph{$T$-equivariant stratification of $X$} consists of a finite partially ordered set $B$ and a collection $\{X_{\beta}\}_{\beta\in B}$ of pairwise disjoint smooth $T$-invariant locally closed subvarieties of $X$ satisfying
\begin{itemize}
\item[(i)] $X=\bigcup_{\beta\in B}X_{\beta}$, and
\item[(ii)] $\overline{X_{\beta}}=\bigcup_{\gamma\leq\beta}X_{\gamma}$ for all $\beta\in B$.
\end{itemize}
\end{definition}

\subsection*{Example} Examples of Definition \ref{def:Finite Stratification} include Bruhat cell decompositions of partial flag varieties. More precisely, suppose that $T_{\mathbb{C}}$ is a maximal torus of a connected, simply-connected complex semisimple group $G$. Suppose further that $T_{\mathbb{C}}\subseteq B\subseteq P$, where $B$ and $P$ are Borel and parabolic subgroups of $G$, respectively. Let $W$ denote the Weyl group and $W_P$ the subgroup of $W$ associated with $P$. One has the partial flag variety $X=G/P$, on which $G$ acts algebraically by left-multiplication. The $T$-fixed points of $X$ are naturally indexed by $W/W_P$. Also, each $B$-orbit contains a unique $T$-fixed point, giving the Bruhat decomposition $$X=\bigsqcup_{u\in W/W_P}BuP/P.$$ For each $u\in W/W_P$, set $X_u:=BuP/P$. Endowing $W/W_P$ with the Bruhat order, one has the closure relations $$\overline{X_{u}}=\bigsqcup_{v\leq u}X_v.$$ Hence, $\{X_u\}_{u\in W/W_P}$ is a $T$-equivariant stratification of $X$.\hfill $\square$ \\

Fix a smooth complex projective $T_{\mathbb C}$-variety $X$ and let $\{X_\beta \}_{\beta \in B}$ be a given equivariant stratification. For each fixed $\beta\in B$, let $N_{\beta}\rightarrow X_{\beta}$ denote the normal bundle of $X_{\beta}$ in $X$ and let $d(\beta)$ denote its rank. The bundle $N_{\beta}$ has a $T$-equivariant Thom class $u_T(\beta)\in \tilde E_T^{2d(\beta)}(Th(N_{\beta}))$ and an associated Euler class $e_T(\beta) \in E_T^{2d(\beta)}(X_\beta)$.


\begin{theorem}\label{thm:Finite Module Isomorphism}
Assume that for each $\beta\in B$, $E_T^*(X_{\beta})$ is a free module over $E_T^*(\text{pt})$, and that $e_T(\beta)$ is not a zero-divisor in $E_T^*(X_{\beta})$. There is an isomorphism
$$E_T^*(X)\cong\bigoplus_{\beta\in B}E_T^*(X_{\beta})$$
of $E_T^*(\text{pt})$-modules.
\end{theorem}

\begin{proof}
Following \cite{Yang-Mills}, we define a subset $J\subseteq B$ to be  \textit{open} if whenever $\beta\in J$ and $\gamma\in B$ satisfy $\beta\leq\gamma$, we have $\gamma\in J$. This definition has the desirable property that if $J\subseteq B$ is open, then $$X_J:=\bigcup_{\beta\in J}X_{\beta}$$ is an open subset of $X$.

Choose a maximal element $\beta_1\in B$ and set $J_1:=\{\beta_1\}$, an open subset of $B$. We inductively define subsets $J_k\subseteq B$, $k\in\{2,\ldots,\vert B\vert\}$, by the condition that $J_k=\{\beta_1,\ldots,\beta_k\}$ with $\beta_k$ a maximal element of $B\setminus J_{k-1}$. By construction, $J_k$ is open for all $k$.

We have graded $E_T^*(\text{pt})$-module isomorphisms
\begin{equation}
	\label{eq:Bundle}
	E_T^*(X_{J_k},X_{J_{k-1}})\cong E_T^*(Th(N_{\beta_k}))\cong E_T^{*-2d(\beta_k)}(X_{\beta_k}),
\end{equation}
the second being the Thom Isomorphism (see \cite{May}, Theorem 9.2). Using \eqref{eq:Bundle}, the long exact sequence of the pair $(X_{J_k},X_{J_{k-1}})$ takes the form
\begin{equation}
	\label{eq:Long-Exact}
	\ldots\rightarrow E_T^{i-2d(\beta_k)}(X_{\beta_k})\xrightarrow\phi E_T^i(X_{J_k})\rightarrow E_T^i(X_{J_{k-1}})\rightarrow E_T^{i-2d(\beta_k)+1}(X_{\beta_k})\rightarrow\ldots.
	\end{equation}

If $E_T^i(X_{J_k}) \to E_T^i(\beta_k)$ is the restriction map, the composition
	$$E_T^{i-2d(\beta_k)}(X_{\beta_k}) \xrightarrow\phi E_T^i(X_{J_k}) \rightarrow E_T^i(X_{\beta_k})$$
is equivalent to multiplication by the equivariant Euler class $e_T(\beta_k)$. As $e_T(\beta_k)$ is not a zero divisor, the composition is injective, forcing $\phi$ to be injective. Hence \eqref{eq:Long-Exact} degenerates to the short exact sequence
\begin{equation}
	\label{eq:Thom-Gysin}
	0 \rightarrow E_T^{*-2d(\beta_k)}(X_{\beta_k})\rightarrow E_T^*(X_{J_k})\rightarrow E_T^*(X_{J_{k-1}})\rightarrow 0
\end{equation}
of $E_T^*(\text{pt})$-modules. Using \eqref{eq:Thom-Gysin} and induction, we will prove that
\begin{equation}
	\label{eq:Induction Claim}
	E_T^*(X_{J_k})\cong\bigoplus_{\ell\leq k}E_T^*(X_{\beta_\ell})
\end{equation} for all $k\in\{2,\ldots,\vert B\vert\}$, {from which the theorem will follow}.

In the base case $k=2$, our short exact sequence is $$0 \rightarrow E_T^{*-2d(\beta_2)}(X_{\beta_2})\rightarrow E_T^*(X_{J_2})\rightarrow E_T^*(X_{\beta_1})\rightarrow 0.$$ This sequence splits by virtue of the fact that $E_T^*(X_{\beta_1})$ is a free $E_T^*(\text{pt})$-module. Hence, $$E_T^*(X_{J_2})\cong E_T^*(X_{\beta_1})\oplus E_T^*(X_{\beta_2}).$$

Assume now that \eqref{eq:Induction Claim} holds for some $k\leq\vert B\vert -1$ and replace $k$ with $k+1$ in \eqref{eq:Thom-Gysin} to obtain the sequence
\begin{equation}
0\rightarrow E_T^{*-2d(\beta_{k+1})}(X_{\beta_{k+1}})\rightarrow E_T^*(X_{J_{k+1}})\rightarrow E_T^*(X_{J_{k}})\rightarrow 0.
\label{eq:SES}
\end{equation}
By assumption, $E_T^*(X_{J_k})$ is free, so \eqref{eq:SES} splits. Hence, \eqref{eq:Induction Claim} holds if we replace $k$ with $k+1$, and our induction is complete.
\end{proof}

\begin{remark}
The isomorphism in Theorem \ref{thm:Finite Module Isomorphism} does not respect the $\mathbb{Z}$-gradings of $E_T^*(X)$ and $\bigoplus_{\beta\in B}E_T^*(X_{\beta})$. To compensate for the degree-shift of $2d(\beta)$ appearing in \eqref{eq:Thom-Gysin}, one can identify $E_T^*(X_{\beta})$ as an $E_T^*(\text{pt})$-module with the principal ideal $\langle e_T(\beta)\rangle$ generated by $e_T(\beta)$. This gives us an isomorphism
\begin{equation}
	\label{eq:Graded Isomorphism}
	E_T^*(X)\cong\bigoplus_{\beta\in B}\langle e_T(\beta)\rangle
\end{equation}
on the level of both $E_T^*(\text{pt})$-modules and $\mathbb{Z}$-graded abelian groups.
\end{remark}


\note{I have moved this up from Section 2.2. Since this is the first section that talks about the Euler class not being a zero divisor, it seems like the conclusion of this section would be a good time to point out that the above Theorem applies to these cohomology theories. Also, add a part concerning $MU_T^*$. This may necessitate changing the lemma to say that $e_T(V)$ is not a zero divisor, for which the particular formulas (i) and (ii) are specifically mentioned in the proofs.}

\subsection{The Case of Finitely Many Fixed Points}\label{sub:Projective Variety Example} 
\label{sub:Finite Fixed Points}

\note{I have completely changed the part of this section prior to the lemma.}
The approach outlined in Section \ref{sec:Cohomology and Stratifications} can be combined with a suitable Bia{\l}ynicki-Birula stratification to yield the $E_T^*$-module structure of a smooth complex projective $T_{\mathbb{C}}$-variety with finitely many $T$-fixed points. More explicitly, we will prove the following theorem:

\begin{theorem}\label{thm:Module Structure}
Suppose that $E_T^*$ is one of $H_T^*$, $K_T^*$, and $MU_T^*$. If $X$ is a smooth complex projective $T_{\mathbb{C}}$-variety with finitely many $T$-fixed points, then $E_T^*(X)$ is a free $E_T^*(\text{pt})$-module of rank $\vert X^T\vert$
\end{theorem}

For the duration of this section, we will assume that everything is as given in the statement of Theorem \ref{thm:Module Structure}.

\begin{lemma}\label{lemma:Coweight}
There exists a coweight $\lambda:\mathbb{C}^*\rightarrow T_{\mathbb{C}}$ with the property that the fixed points of the resulting $\mathbb{C}^*$-action on $X$ are precisely the $T$-fixed points.
\end{lemma}

\begin{proof}
Choose a coweight $\lambda$ such that for every $w\in X^T$ and weight $\mu:T_{\mathbb{C}}\rightarrow\mathbb{C}^*$ of the isotropy representation $T_wX$, the pairing $\langle\lambda,\mu\rangle$ is non-zero. This coweight yields an algebraic action of $\mathbb{C}^*$ on $X$, and we suppose that $Y$ is an irreducible component of $X^{\mathbb{C}^*}$. Note that $Y$ is a smooth closed $T_{\mathbb{C}}$-invariant subvariety of $X$. 
By the Borel Fixed Point Theorem, $Y$ has a $T$-fixed point $y$. Since $T_yY$ is precisely the trivial weight space of the $\mathbb{C}^*$-representation on $T_yX$, our choice of $\lambda$ implies that $T_yY=\{0\}$. It follows that $Y=\{y\}$, giving the inclusion $X^{\mathbb{C}^*}\subseteq X^T$.
\end{proof}

Now, select $\lambda:\mathbb{C}^*\rightarrow T_{\mathbb{C}}$ as in Lemma \ref{lemma:Coweight}. 
Given $w\in X^{\mathbb{C}^*}=X^T$, one has the smooth locally closed subvariety
\begin{equation}
	\label{eq:BB definition}
	X_w:=\left\{x\in X:\lim_{t\to 0}(\lambda(t)\cdot x)=w\right\}.
\end{equation}
The $X_w$ constitute a Bia{\l}ynicki-Birula stratification \cite{Bialynicki-Birula1973}, a $T$-equivariant stratification of $X$. Furthermore, $X_w$ is $T$-equivariantly homeomorphic to the $T$-submodule $(T_wX)^+$ of $T_wX$ spanned by the weight vectors whose weights have strictly positive pairing with $\lambda$. In particular, $X_w$ equivariantly retracts onto its $T$-fixed point $\{w\}$ and we have a ring isomorphism $r_w:E_T^*(X_w)\xrightarrow{\cong}E_T^*(\{w\})$. If $e_T(w)\in E_T^*(X_w)$ denotes the $T$-equivariant Euler class of the normal bundle of $X_w$ in $X$, then $r_w(e_T(w))$ is the $T$-equivariant Euler class of the quotient representation $T_{w}(X)/T_{w}X_w\rightarrow\{w\}$.

\begin{lemma}\label{lemma:Formula}
Let $V$ be a finite-dimensional complex $T$-representation such that $V^T = \{ 0\}$, viewed as a $T$-equivariant vector bundle over a point. If $E_T^*$ is $H_T^*, K_T^*$, or $MU_T^*$, then the $T$-equivariant Euler class $e_T(V) \in E_T^*(\text{pt})$ is not a zero divisor.
\end{lemma}

\begin{proof}
Note that $E_T^*(\text{pt})$ is an integral domain for each of the above three theories. By virtue of the Whitney sum formula, it therefore suffices to prove that $e_T(V)$ is non-zero when $V$ is one-dimensional.

Let $\mu\in X^*(T)$ be the (non-zero) weight of $V$. If $E_T^*=H_T^*$, then $e_T(V)$ is the ordinary Euler class of the associated bundle $ET\times_TV\rightarrow BT$. Under the usual ring isomorphism $H^*(BT;\mathbb{Z})\cong\Sym_{\mathbb{Z}}(X^*(T))$, this Euler class corresponds to the weight $\mu$.

When $E_T^*=K_T^*$, the equivariant Euler class of a complex $T$-representation is given by the alternating sum of its exterior powers in $K_T^*(\text{pt})$\cite[Chapter XIV, Theorem 3.2]{May}. Hence, $e_T(V)=1-[V]\in K_T^2(\text{pt})$, which is identified with $1-e^{\mu}$ under the isomorphism $K_T^2(\text{pt})\cong R(T)$. We thus see that $e_T(V)\neq 0$.

In the case of $MU_T^*$, we simply appeal to \cite{Sinha}.
\end{proof}

Since the $T$-fixed points in $X$ are isolated, zero is not a weight of the representation $T_{w}X/T_{w}X_w$.  By Lemma \ref{lemma:Formula}, we conclude that $r_w(e_T(w))$ is not a zero-divisor in $E_T^*(\{w\})$, meaning that $e_T(w)$ is not a zero divisor. An application of Theorem \ref{thm:Finite Module Isomorphism} then yields an $E_T^*(\text{pt})$-module isomorphism
$$E_T^*(X)\cong\bigoplus_{w\in X^T}E_T^*(X_w).$$
In particular, $E_T^*(X)$ is free of rank $\vert X^T\vert$, proving Theorem \ref{thm:Module Structure}.


Theorem \ref{thm:Module Structure} will prove essential in extending our results to the case of direct limits of projective varieties. To realize the extension, we will require the following lemma.

\begin{proposition}\label{prop:Relative}
If $Y$ is a smooth closed $T_{\mathbb{C}}$-invariant subvariety of $X$, then
\begin{itemize}
\item[(i)] $E_T^*(X,Y)$ is a free $E_T^*(\text{pt})$-module of finite rank vanishing in odd grading degrees, and
\item[(ii)] the restriction map $E_T^*(X)\rightarrow E_T^*(Y)$ is surjective.
\end{itemize}
\end{proposition}

\begin{proof}
To prove (i), we will appeal to some general properties of model categories. \note{We might include an introductory reference.} Indeed, $T$-spaces form a model category in which the weak equivalences are the $T$-homotopy equivalences and the cofibrations are the morphisms with the $T$-homotopy extension property. Accordingly, we will begin by proving the following claim by induction: If $w_1,\ldots,w_n\in Y^T$ {and $X_{w_i}$ are the associated Bia\l ynicki-Birula strata}, then the inclusion
$$Y\rightarrow Y\cup\bigcup_{i=1}^nX_{w_i}$$
is an acyclic cofibration (ie. a cofibration that is also a weak equivalence). \note{This requires a reference. At the moment, the only reference is Peter May's response on Math Overflow.}

For the base case, let $Y_{w_1}\subseteq Y$ denote the Bia{\l}ynicki-Birula stratum of $Y$ associated with $w_1\in Y^T$. One has the pushout square
$$\xymatrix{
Y_{w_1} \ar[d] \ar[r] & X_{w_1} \ar[d] \\
Y \ar[r] & Y\cup X_{w_1}}$$
of inclusions. Note that $Y_{w_1}\rightarrow X_{w_1}$ is an acyclic cofibration. \note{This corresponds to an inclusion of finite-dimensional complex $T$-representations. Therefore, $(Y_w,X_w)$ is a $T$-NDR pair and $Y_w\rightarrow X_w$ is a $T$-cofibration. A reference is page 504 of "On $G$-ANR's and their $G$-Homotopy Types" by Murayama.}
Since the pushout of an acyclic cofibration is itself an acyclic cofibration, it follows that $Y\rightarrow Y\cup X_{w_1}$ is an acyclic cofibration. Now, assume that our claim holds for $\leq n$ points in $Y^T$. Given $w_1,\ldots,w_{n+1}\in Y^T$, we consider the pushout square $$\xymatrix{
Y \ar[d]^{i_2} \ar[r]^{i_1} & Y\cup\bigcup_{i=1}^nX_{w_i} \ar[d]^{j_2} \\
Y\cup X_{w_{n+1}} \ar[r]^{j_1} & Y\cup\bigcup_{i=1}^{n+1}X_{w_i}}$$ of inclusions. Noting that $i_1$ is an acyclic cofibration, the same is true of $j_1$. The inclusion $Y\rightarrow Y\cup\bigcup_{i=1}^{n+1}X_{w_i}$ is then a composition of the acyclic cofibrations $i_2$ and $j_1$, and so is itself an acyclic cofibration. This completes the induction. Setting $$Z:=\bigcup_{w\in Y^T}X_{w},$$ it follows that $Y\rightarrow Z$ is an acyclic cofibration. In particular, $E_T^*(Z,Y)=0$, and it just remains to prove that $E_T^*(X,Z)$ is free of finite rank and vanishes in odd degrees. \note{This would seem to follow from the G-Whitehead Theorem. A weak equivalence of G-CW complexes is a G-homotopy equivalence. Actually, we need not invoke this since our model-theoretic notion of a weak equivalence is precisely a $T$-homotopy equivalence.}

Recall that if $w\in X^T$, then $X_w$ is $T$-equivariantly homeomorphic to a finite-dimensional complex $T$-representation $V_w$. Choose an enumeration $\{w_1,\ldots,w_m\}$ of $X^T\setminus Y^T$ with the property that for all $k\in\{1,\ldots,m\}$, the quotient of $Z\cup\bigcup_{j=1}^kX_{w_j}$ by $Z\cup\bigcup_{j=1}^{k-1}X_{w_j}$ is $T$-equivariantly homeomorphic to the one-point compactification $S^{V_{w_k}}$. \note{Perhaps we must assume $X$ to be connected. Also, we should explain why this is possible.}
Using induction, we will prove that $E_T^*(Z\cup\bigcup_{j=1}^kX_{w_j},Z)$ is free of finite rank for all $k\in\{1,\ldots,m\}$, and that it vanishes in odd grading degrees.

Since $Z\cap X_{w_1}=\emptyset$, the inclusion $Z\rightarrow Z\cup X_{w_1}$ is a cofibration. \note{To see this, create the pushout square with $\emptyset$ in the upper left, $X_{w_1}$ and $Z$ on the off diagonal, and $Z\cup X_w$ on the bottom right}
Hence,
$$E_T^*\left(Z\cup X_{w_1},Z\right)\cong\tilde{E}_T^*\left((Z\cup X_{w_1})/Z\right)\cong\tilde{E}_T^*\left(S^{V_{w_k}}\right)$$
is free of finite rank, and vanishes in odd grading degrees. Now, assume that $E_T^*\left(Z\cup\bigcup_{j=1}^k X_{w_j},Z\right)$ vanishes in odd degrees and is free of finite rank. Since the inclusion $Z\cup\bigcup_{j=1}^kX_{w_j}\rightarrow Z\cup\bigcup_{j=1}^{k+1}X_{w_j}$ is a cofibration, we find that \note{For the next few lines, either the encompassing brackets need to be made bigger, or we must not use bigcup}
$$E_T^*\left(Z\cup\bigcup_{j=1}^{k+1}X_{w_j},Z\cup\bigcup_{j=1}^kX_{w_j}\right)\cong\tilde{E}_T^*\left(\bigg(Z\cup\bigcup_{j=1}^{k+1}X_{w_j}\bigg)\big/\bigg(Z\cup\bigcup_{j=1}^kX_{w_j}\bigg)\right)\cong\tilde{E}_T^*\left(S^{V_{w_{k+1}}}\right)$$
is also free of finite rank and vanishes in odd degrees. \note{Consider the pushout square with the empty set in the upper-left, $X_{w_{k+1}}$ and $Z\cup\bigcup_{j=1}^kX_{w_j}$ on the off-diagonal, and the union in the lower-right.}
Therefore, the long exact sequence of the pairs $(Z\cup\bigcup_{j=1}^{k+1}X_{w_j},Z\cup\bigcup_{j=1}^{k}X_{w_j})$, $(Z\cup\bigcup_{j=1}^{k+1}X_{w_j}, Z)$, $(Z\cup\bigcup_{j=1}^{k}X_{w_j}, Z)$ splits to give the short exact sequence
$$0\rightarrow E_T^*\left(Z\cup\bigcup_{j=1}^{k+1}X_{w_j},Z\cup\bigcup_{j=1}^{k}X_{w_j}\right)\rightarrow E_T^*\left(Z\cup\bigcup_{j=1}^{k+1}X_{w_j},Z\right)\rightarrow E_T^*\left(Z\cup\bigcup_{j=1}^{k}X_{w_j},Z\right)\rightarrow 0.$$
Since $E_T^*(Z\cup\bigcup_{j=1}^{k+1}X_{w_j},Z\cup\bigcup_{j=1}^{k}X_{w_j})$ and $E_T^*(Z\cup\bigcup_{j=1}^{k}X_{w_j},Z)$ are free of finite rank, the same is true of $E_T^*(Z\cup\bigcup_{j=1}^{k+1}X_{w_j},Z)$. We have therefore proved (i).

For (ii), we consider the long exact sequence of the pair $(X,Y)$. Indeed, (i) is then seen to imply that $E_T^n(X)\rightarrow E_T^n(Y)$ is surjective for even $n$. Furthermore, the isomorphism \eqref{eq:Graded Isomorphism} establishes that both $E_T^*(X)$ and $E_T^*(Y)$ vanish in odd grading degrees. The proof is therefore complete.

\end{proof}


\subsection{Direct Limits of Projective Varieties} 
\label{sub:Direct Limits of Projective Varieties}

We now provide a generalization of our findings in Section \ref{sub:Projective Variety Example}, replacing projective varieties with direct limits thereof. As before, $T$ denotes a compact torus with complexification $T_{\mathbb{C}}$,
 and $E_T^*$ is one of $H_T^*$, $K_T^*$, and $MU_T^*$. Suppose that $$X_0\subseteq X_1\subseteq X_2\subseteq\ldots\subseteq X_n\subseteq\ldots$$ is a sequence of equivariant closed embeddings of smooth complex projective $T_{\mathbb{C}}$-varieties with $(X_n)^T$ finite for each $n\geq 0$. Let $X$ be the topological direct limit of the $X_n$ in their analytic topologies, and endow $X$ with the induced direct limit topology. Note that $X$ then carries a continuous action of $T$. The following theorem then generalizes Theorem \ref{thm:Module Structure}:

\begin{theorem}\label{thm:Direct Limit}
Under the conditions stated above, there is an $E_T^*(\text{pt})$-module isomorphism $$E_T^*(X)\cong\prod_{x\in X^T}E_T^*(\text{pt}).$$
\end{theorem}

\begin{proof}
By Proposition \ref{prop:Relative}, each restriction map $E_T^*(X_{n+1})\rightarrow E_T^*(X_n)$ is surjective. Hence, the inverse system $\{E_T^*(X_n)\}_n$ of $E_T^*(\text{pt})$-modules has vanishing Milnor $\varprojlim^1$. It follows that the canonical map $E_T^*(X)\rightarrow\varprojlim_nE_T^*(X_n)$ is an isomorphism\cite{Axiomatic}.

 It will therefore suffice to prove that $\{E_T^*(X_n)\}_n$ and $\{\bigoplus_{x\in (X_n)^T}E_T^*(\text{pt})\}_n$ are isomorphic as inverse systems of $E_T^*(\text{pt})$-modules, where the maps in the latter system are precisely the projection maps resulting from the inclusions $(X_n)^T\subseteq (X_{n+1})^T$. We will do this by inductively constructing $E_T^*(\text{pt})$-module isomorphisms $$\psi_n:E_T^*(X_n)\rightarrow\bigoplus_{x\in (X_n)^T}E_T^*(\text{pt})$$ making the diagrams
$$D_n:=\xymatrix{
        E_T^*(X_{n+1}) \ar[d] \ar[r]^-{\psi_{n+1}} & \bigoplus_{x\in (X_{n+1})^T}E_T^*(\text{pt}) \ar[d]\\
        E_T^*(X_n) \ar[r]^-{\psi_{n}} & \bigoplus_{x\in (X_n)^T}E_T^*(\text{pt})}$$
commute.

By Theorem \ref{thm:Module Structure}, we haves an $E_T^*(\text{pt})$-module isomorphism $\psi_0:E_T^*(X_0)\rightarrow\bigoplus_{x\in (X_0)^T}E_T^*(\text{pt})$.
Assume now that we have constructed isomorphisms $\psi_k:E_T^*(X_k)\rightarrow\bigoplus_{x\in (X_k)^T}E_T^*(\text{pt})$ for all $k\leq n$ so that the diagrams $D_0,\ldots,D_{n-1}$ commute. Since the restriction $\pi_n:E_T^*(X_{n+1})\rightarrow E_T^*(X_n)$ is surjective, the long exact sequence of the pair $(X_{n+1},X_n)$ degenerates to a short exact sequence
\begin{equation}
\label{eq:Short Exact Sequence}
0\rightarrow E_T^*(X_{n+1},X_n)\rightarrow E_T^*(X_{n+1})\xrightarrow{\pi_n} E_T^*(X_n)\rightarrow 0
\end{equation}
of $E_T^*(\text{pt})$-modules. Theorem \ref{thm:Module Structure} implies that $E_T^*(X_n)$ is free, so that \eqref{eq:Short Exact Sequence} admits a splitting $\varphi_n:E_T^*(X_{n+1})\rightarrow E_T^*(X_{n+1},X_n)$. Also, Proposition \ref{prop:Relative} implies that $E_T^*(X_{n+1},X_n)$ is free of rank $\vert (X_{n+1})^T\setminus (X_n)^T\vert$. We may therefore choose an $E_T^*(\text{pt})$-module isomorphism
$$\theta_n:E_T^*(X_{n+1},X_n)\xrightarrow{\cong}\bigoplus_{x\in (X_{n+1})^T\setminus (X_n)^T}E_T^*(\text{pt}).$$
The composite map
$$E_T^*(X_{n+1})\xrightarrow{(\pi_n,\varphi_n)} E_T^*(X_n)\oplus E_T^*(X_{n+1},X_n)\xrightarrow{\psi_n\oplus\theta_n}\bigoplus_{x\in (X_{n+1})^T}E_T^*(\text{pt})$$
is then an $E_T^*(\text{pt})$-module isomorphism that we shall call $\psi_{n+1}$. By construction, $D_n$ commutes for this choice of $\psi_{n+1}$, and this completes the proof.
\end{proof}

\note{I have moved the first two sections on the affine Grassmannian so that they now follow Sections 1 and 2.}

\section{The Affine Grassmannian} 
\label{sec: Affine Grassmannian}
	The affine Grassmannian $\mathcal Gr$ is a space of great interest to geometric representation theorists (see \cite{Kamnitzer,MV}, for instance). It is also very closely linked to the study of (algebraic) based loop groups (discussed in \cite{Mitchell,Magyar,PressleySegal}). Using the work done in the aforementioned papers, we can show that $\mathcal Gr$ is the perfect candidate for an application of Theorem \ref{thm:Direct Limit}.

\subsection{Definition and Filtration} 
\label{sub:Definition and Filtration}

Let $G$ be a connected, simply-connected complex semisimple group. Fix a maximal torus $T_{\mathbb{C}}\subseteq G$ with compact real from $T_{\mathbb R}$, as well as a Borel subgroup $B$ containing $T_{\mathbb C}$. Take $W= N_G(T_{\mathbb C})/T_{\mathbb C}$ to be the associated Weyl group.

Let $X^*(T_{\mathbb{C}}):=\Hom(T_{\mathbb{C}},\mathbb{C}^*)$ and $X_*(T_{\mathbb{C}}):=\Hom(\mathbb{C}^*,T_{\mathbb C})$ be the weight and coweight lattices respectively, endowed with their usual pairing
$$\langle\cdot,\cdot\rangle:X_*(T_{\mathbb{C}})\otimes_{\mathbb{Z}}X^*(T_{\mathbb{C}})\rightarrow\mathbb{Z}.$$
The choice of Borel subgroup yields dominant weights $X^*(T_{\mathbb{C}})_+\subseteq X^*(T_{\mathbb{C}})$ and dominant coweights $X_*(T_{\mathbb{C}})_+\subseteq X_*(T_{\mathbb{C}})$. Take $\Delta\subseteq X^*(T_{\mathbb{C}})$ to be the collection of roots, and $\Pi\subseteq\Delta$ to be the subset of simple (positive) roots.

We shall assume that $G$ admits a finite-dimensional, faithful, irreducible representation $V(\alpha)$ of highest weight $\alpha\in X^*(T_{\mathbb{C}})_+$. This allows us to realize $G$ as a Zariski-closed subgroup of $GL(V(\alpha))$.

Consider the $\mathbb{C}$-algebras $\mathcal{O}:=\mathbb{C}[t]$ and $\mathcal{K}:=\mathbb{C}[t,t^{-1}]$, letting $G(\mathcal{O})$ and $G(\mathcal{K})$ denote the $\mathcal{O}$ and $\mathcal{K}$-valued points of $G$, respectively. Set-theoretically, the affine Grassmannian of $G$ is defined to be the coset space
$$\mathcal{G}r:=G(\mathcal{K})/G(\mathcal{O}).$$

Note that the $\mathbb{C}$-vector space $V(\alpha)\otimes\mathcal{K}$ admits the filtration
$$\ldots\subseteq V(\alpha)\otimes t^2\mathcal{O}\subseteq V(\alpha)\otimes t\mathcal{O}\subseteq V(\alpha)\otimes\mathcal{O}\subseteq V(\alpha)\otimes t^{-1}\mathcal{O}\subseteq V(\alpha)\otimes t^{-2}\mathcal{O}\subseteq\ldots.$$
We thus define a function $\text{val}:V(\alpha)\otimes\mathcal{K}\rightarrow\mathbb{Z}$ by
$$\text{val}(u):=\max\{k\in\mathbb{Z}:u\in V(\alpha)\otimes t^{k}\mathcal{O}\}.$$
As $G(\mathcal{K})$ acts on $V(\alpha)\otimes\mathcal{K}$ by virtue of the inclusion of $G$ into $\GL(V(\alpha))$, we may define $\text{Val}:G(\mathcal{K})\rightarrow\mathbb{Z}$ by
$$\text{Val}(g):=\min\{\text{val}(g\cdot v):v\in V(\alpha)\}.$$

Given $n\in\mathbb{Z}_{\geq 0}$, we set
$$G(\mathcal{K})_n:=\{g\in G(\mathcal{K}):\text{Val}(g)\geq -n\},$$
yielding a filtration
\begin{equation}
\label{eq:AffGeoFilt}
G(\mathcal{O})=G(\mathcal{K})_0\subseteq G(\mathcal{K})_1\subseteq G(\mathcal{K})_2\subseteq\ldots\subseteq G(\mathcal{K})
\end{equation}
of $G(\mathcal{K})$. Note that $G(\mathcal{K})_n$ is invariant under the right-multiplicative action of $G(\mathcal{O})$ on $G(\mathcal{K})$. Accordingly, we define
\begin{equation}
\label{eq:GrFilt}
\mathcal{G}r_n:=G(\mathcal{K})_n/G(\mathcal{O}),
\end{equation}
a smooth finite-dimensional projective scheme over $\mathbb{C}$. By exhibiting the affine Grassmannian as inductive limit of the schemes $\{\mathcal{G}r_n\}_{n\in\mathbb{Z}_{\geq 0}}$, we may realize $\mathcal{G}r$ as a projective ind-scheme. (For a treatment of ind-schemes, the reader might consult the appendix of \cite{FB}.)

Of course, we will endow $\mathcal{G}r$ with a topology other than the one it inherits as an ind-scheme. Namely, we will regard $\mathcal{G}r$ as the direct limit of the varieties $\{\mathcal{G}r_n\}_{n\in\mathbb{Z}_{\geq 0}}$ in their classical topologies.


\subsection{The Action of $\mathbb{C}^*$} 
\label{sub:Action of C*}

There is a natural ``loop rotation'' action of $\mathbb{C}^*$ on $G(\mathcal{K})$. Indeed, the left-multiplicative action of $\mathbb{C}^*$ on itself yields a $\mathbb{C}^*$-action on $\Hom(\mathbb{C}^*,G)=G(\mathcal{K})$ by group automorphisms. More concretely, the inclusion $G\subseteq\GL(V(\alpha))$ associates to each point $\gamma\in G(\mathcal{K})$ an expansion $\gamma=\sum_{j\in\mathbb{Z}}\gamma_jt^j$, where $\gamma_j\in\End(V(\alpha))$ for all $j$. The action of $s\in\mathbb{C}^*$ on $\gamma$ is then given by $$s:\sum_{j\in\mathbb{Z}}\gamma_jt^j\mapsto\sum_{j\in\mathbb{Z}}\gamma_j(st)^j.$$ It follows that $G(\mathcal{K})_n$ is $\mathbb{C}^*$-invariant for $n\in\mathbb{Z}_{\geq 0}$. In particular, $\mathcal{G}r_0=G(\mathcal{O})$ is $\mathbb{C}^*$-invariant and the $\mathbb{C}^*$-action descends to an action on $\mathcal{G}r$ that preserves each subvariety $\mathcal{G}r_n$.


\subsection{The Generalized Torus-Equivariant Cohomology of $\mathcal{G}r$} 
\label{sub:A Stratification}

Consider the compact torus $T:=T_{\mathbb{R}}\times S^1$, where $T_{\mathbb{R}}$ is the compact torus fixed in Section \ref{sub:Definition and Filtration}. As $T$ is a subgroup of $T_{\mathbb{C}}\times\mathbb{C}^*$, and the latter torus acts on $\mathcal{G}r$ via the commuting actions of $T_{\mathbb{C}}$ and $\mathbb{C}^*$, we obtain an action of $T$ on $\mathcal{G}r$.

Note that \eqref{eq:GrFilt} is thus a $T$-equivariant filtration. With Theorem \ref{thm:Direct Limit} in mind, it remains only to prove that $(\mathcal{G}r_n)^T$ is finite for all $n\geq 0$. To this end, let $\lambda\in X_*(T_{\mathbb{C}})$ be a coweight, and consider the point in $G(\mathcal{K})$ given by the composition
\begin{equation}
\label{eq:lambdaPoint}
\mathbb{C}^*\xrightarrow{\lambda} T_{\mathbb{C}}\hookrightarrow G,
\end{equation}
where $T_{\mathbb{C}}\hookrightarrow G$ is the inclusion. Let  $t^{\lambda}\in\mathcal{G}r$ denote the class of \eqref{eq:lambdaPoint} in the affine Grassmannian. It turns out (see \cite{MV}) that the $T$-fixed points of $\mathcal{G}r$ are precisely the $t^{\lambda}$, for $\lambda\in X_*(T_{\mathbb{C}})$, leading us to prove the following lemma:

\begin{lemma}
For $n\geq 0$, $$(\mathcal{G}r_n)^T=\{t^{w\lambda}:w\in W, \text{ }\lambda\in X_*(T_{\mathbb{C}})_+, \langle\lambda,w_0\alpha\rangle\geq -n\},$$ where $w_0\in W$ is the longest element. In particular, $(\mathcal{G}r_n)^T$ is finite.
\end{lemma}

\begin{proof}
Since $\mathcal{G}r_n$ is $G$-invariant, one has an induced action of $W$ on $(\mathcal{G}r_n)^{T_{\mathbb{C}}}$. Because the actions of $G$ and $\mathbb{C}^*$ commute, the $W$-action leaves $(\mathcal{G}r_n)^{T_{\mathbb{C}}\times\mathbb{C}^*}=(\mathcal{G}r_n)^T$ invariant. Hence, if $\mu\in X_*(T_{\mathbb{C}})$ is $W$-conjugate to $\lambda\in X_*(T_{\mathbb{C}})_+$, then $t^{\mu}\in(\mathcal{G}r_n)^T$ if and only if $t^{\lambda}\in(\mathcal{G}r_n)^T$. Our task is therefore to prove that if $\lambda\in X_*(T_{\mathbb{C}})_+$, then $t^{\lambda}\in(\mathcal{G}r_n)^T$ if and only if $\langle\lambda,w_0\alpha\rangle\geq -n$.

Suppose that $\lambda\in X_*(T_{\mathbb{C}})_+$, and let $v\in V(\alpha)$ be a vector of weight $\xi\in X^*(T_{\mathbb{C}})$. Note that for all $t\in\mathbb{C}^*$, $$\lambda(t)\cdot v=\xi(\lambda(t))v=t^{\langle\lambda,\xi\rangle}v.$$ Hence, if we regard $\lambda$ as a point in $G(\mathcal{K})$, then $$\lambda\cdot v=v\otimes t^{\langle\lambda,\xi\rangle}\in V(\alpha)\otimes t^{\langle\lambda,\xi\rangle}\mathcal{O}.$$ Since $V(\alpha)$ has a basis of weight vectors, it follows that $\text{Val}(\lambda)$ is the minimum of $\langle\lambda,\xi\rangle$, where $\xi$ ranges over the weights of $V(\alpha)$. Noting that $w_0\alpha$ is the lowest weight of $V(\alpha)$, we conclude that $\text{Val}(\lambda)=\langle\lambda,w_0\alpha\rangle$. Therefore, $\lambda\in G(\mathcal{K})_n$ if and only if $\langle\lambda,w_0\alpha\rangle\geq -n$. This completes the proof.
\end{proof}

We may thus apply Theorem \ref{thm:Direct Limit} to compute the module structure of $E_T^*(\mathcal{G}r)$ for $E_T^*=H_T^*$, $K_T^*$, or $MU_T^*$. Indeed, we have $$E_T^*(\mathcal{G}r)\cong\prod_{\lambda\in X_*(T_{\mathbb{C}})}E_T^*(\text{pt})$$ as $E_T^*(\text{pt})$-modules.



\bibliographystyle{plain}
\bibliography{GenEqBib}

\begin{thebibliography}{10}

\bibitem{Yang-Mills}
M.~F. Atiyah and R.~Bott.
\newblock The {Yang-Mills} equations over {Riemann} surfaces.
\newblock {\em Philosophical Transactions of the Royal Society of London.
  Series A, Mathematical and Physical Sciences}, 308(1505):pp. 523--615, 1983.

\bibitem{Bialynicki-Birula1973}
A.~Bia{\l}ynicki-Birula.
\newblock Some theorems on actions of algebraic groups.
\newblock {\em Annals of Mathematics}, 98(3):pp. 480--497, 1973.

\bibitem{CGK}
Michael Cole, J.~P.~C. Greenlees, and I.~Kriz.
\newblock The universality of equivariant complex bordism.
\newblock {\em Math. Z.}, 239(3):455--475, 2002.

\bibitem{FB}
Edward Frenkel and David Ben-Zvi.
\newblock {\em Vertex algebras and algebraic curves}, volume~88 of {\em
  Mathematical Surveys and Monographs}.
\newblock American Mathematical Society, Providence, RI, second edition, 2004.

\bibitem{HHH2005}
Megumi Harada, Andr{\'e} Henriques, and Tara~S. Holm.
\newblock Computation of generalized equivariant cohomologies of {K}ac-{M}oody
  flag varieties.
\newblock {\em Adv. Math.}, 197(1):198--221, 2005.

\bibitem{Axiomatic}
Mark Hovey, John~H. Palmieri, and Neil~P. Strickland.
\newblock Axiomatic stable homotopy theory.
\newblock {\em Mem. Amer. Math. Soc.}, 128(610):x+114, 1997.

\bibitem{Kamnitzer}
Joel Kamnitzer.
\newblock Mirkovi\'c-{V}ilonen cycles and polytopes.
\newblock {\em Ann. of Math. (2)}, 171(1):245--294, 2010.

\bibitem{KirwanBook}
Frances~Clare Kirwan.
\newblock {\em Cohomology of quotients in symplectic and algebraic geometry},
  volume~31 of {\em Mathematical Notes}.
\newblock Princeton University Press, Princeton, NJ, 1984.

\bibitem{Magyar}
Peter Magyar.
\newblock Schubert classes of a loop group.

\bibitem{May}
J.~P. May.
\newblock {\em Equivariant homotopy and cohomology theory}, volume~91 of {\em
  CBMS Regional Conference Series in Mathematics}.
\newblock Published for the Conference Board of the Mathematical Sciences,
  Washington, DC; by the American Mathematical Society, Providence, RI, 1996.
\newblock With contributions by M. Cole, G. Comeza{\~n}a, S. Costenoble, A. D.
  Elmendorf, J. P. C. Greenlees, L. G. Lewis, Jr., R. J. Piacenza, G.
  Triantafillou, and S. Waner.

\bibitem{MV}
I.~Mirkovi{\'c} and K.~Vilonen.
\newblock Geometric {L}anglands duality and representations of algebraic groups
  over commutative rings.
\newblock {\em Ann. of Math. (2)}, 166(1):95--143, 2007.

\bibitem{Mitchell}
Stephen~A. Mitchell.
\newblock Quillen's theorem on buildings and the loops on a symmetric space.
\newblock {\em Enseign. Math. (2)}, 34(1-2):123--166, 1988.

\bibitem{PressleySegal}
Andrew Pressley and Graeme Segal.
\newblock {\em Loop groups}.
\newblock Oxford Mathematical Monographs. The Clarendon Press, Oxford
  University Press, New York, 1986.
\newblock Oxford Science Publications.

\bibitem{Segal}
Graeme Segal.
\newblock Equivariant {$K$}-theory.
\newblock {\em Inst. Hautes \'Etudes Sci. Publ. Math.}, (34):129--151, 1968.

\bibitem{Sinha}
Dev~P. Sinha.
\newblock Computations of complex equivariant bordism rings.
\newblock {\em Amer. J. Math.}, 123(4):577--605, 2001.

\end{thebibliography}

\end{document}